\documentclass[11pt]{article}
\usepackage{hyperref}
\usepackage{enumerate,fullpage}
\usepackage{amssymb,amsmath,graphicx,amsthm}
\usepackage{color}

\usepackage{amsfonts,amsthm, amsmath}
\usepackage{epsfig}
\usepackage{makeidx}
\usepackage{graphicx,epstopdf}
\DeclareGraphicsExtensions{.ps,.png,.jpg}
\usepackage{enumerate}

\makeindex
\usepackage{subfig}
\usepackage{mathrsfs}

\def\E{{\mathbb E}}

\newcommand{\ncom}{\newcommand}
\ncom{\ul}{\underline}
\ncom{\beq}{\begin{equation}}
\ncom{\eeq}{\end{equation}}
\ncom{\bea}{\begin{eqnarray*}}
\ncom{\eea}{\end{eqnarray*}}
\ncom{\beqa}{\begin{eqnarray}}
\ncom{\eeqa}{\end{eqnarray}}
\ncom{\nno}{\nonumber}
\ncom{\non}{\nonumber}
\ncom{\ds}{\displaystyle}
\ncom{\half}{\frac{1}{2}}
\ncom{\mbx}{\makebox{.25cm}}
\ncom{\hs}{\mbox{\hspace{.25cm}}}
\ncom{\rar}{\rightarrow}
\ncom{\Rar}{\Rightarrow}
\ncom{\noin}{\noindent}
\ncom{\bc}{\begin{center}}
\ncom{\ec}{\end{center}}
\ncom{\sz}{\scriptsize}
\ncom{\rf}{\ref}
\ncom{\s}{\sqrt{2}}
\ncom{\sgm}{\sigma}
\ncom{\Sgm}{\Sigma}
\ncom{\psgm}{\sigma^{\prime}}
\ncom{\dt}{\delta}
\ncom{\Dt}{\Delta}
\ncom{\lmd}{\lambda}
\ncom{\Lmd}{\Lambda}
\ncom{\Th}{\Theta}
\ncom{\e}{\eta}
\ncom{\eps}{\epsilon}
\ncom{\pcc}{\stackrel{P}{>}}
\ncom{\lp}{\stackrel{L_{p}}{>}}
\ncom{\dist}{{\rm\,dist}}
\ncom{\sspan}{{\rm\,span}}
\ncom{\re}{{\rm Re\,}}
\ncom{\im}{{\rm Im\,}}
\ncom{\sgn}{{\rm sgn\,}}
\ncom{\ba}{\begin{array}}
\ncom{\ea}{\end{array}}
\ncom{\hone}{\mbox{\hspace{1em}}}
\ncom{\htwo}{\mbox{\hspace{2em}}}
\ncom{\hthree}{\mbox{\hspace{3em}}}
\ncom{\hfour}{\mbox{\hspace{4em}}}
\ncom{\vone}{\vskip 2ex}
\ncom{\vtwo}{\vskip 4ex}
\ncom{\vonee}{\vskip 1.5ex}
\ncom{\vthree}{\vskip 6ex}
\ncom{\vfour}{\vspace*{8ex}}
\ncom{\norm}{\|\;\;\|}
\ncom{\integ}[4]{\int_{#1}^{#2}\,{#3}\,d{#4}}
\ncom{\vspan}[1]{{{\rm\,span}\{ #1 \}}}
\ncom{\dm}[1]{ {\displaystyle{#1} } }
\ncom{\ri}[1]{{#1} \index{#1}}
\newtheorem{theorem}{\bf Theorem}[section]
\newtheorem{remark}{\bf Remark}[section]
\newtheorem{proposition}{Proposition}[section]

\newtheorem{corollary}{Corollary}[section]

\newtheorem{definition}{Definition}[section]
\newtheoremstyle
    {remarkstyle}
    {}
    {11pt}
    {}
    {}
    {\bfseries}
    {:}
    {     }
    {\thmname{#1} \thmnumber{#2} }

\theoremstyle{remarkstyle}

\begin{document}

\newpage

\begin{center}
{\Large \bf Fractional Poisson Processes of Order $k$ and Beyond}
\end{center}
\vone
\vone
\begin{center}
{Neha Gupta}$^{\textrm{a}}$, {Arun Kumar}$^{\textrm{a}}$
\footnotesize{
		$$\begin{tabular}{l}
		$^{\textrm{a}}$ \emph{Department of Mathematics, Indian Institute of Technology Ropar, Rupnagar, Punjab - 140001, India}
\end{tabular}$$}
\end{center}
\vtwo
\begin{center}
\noindent{\bf Abstract}
\end{center}
In this article, we introduce fractional Poisson fields of order $k$ in n-dimensional Euclidean space $\mathbb{R}^{n}_{+}$. We also work on time-fractional Poisson process of order $k$, space-fractional Poisson process of order $k$ and tempered version of  time-space fractional Poisson process of order $k$ in one dimensional Euclidean space $\mathbb{R}^{1}_{+}$. These processes are defined in terms of fractional compound Poisson processes. Time-fractional Poisson process of order $k$ naturally generalizes the Poisson process and Poisson process of order $k$ to a heavy tailed waiting times counting process. The space-fractional Poisson process of order $k$, allows on average infinite number of arrivals in any interval. We derive the marginal probabilities, governing difference-differential equations of the introduced processes. We also provide Watanabe martingale characterization for some time-changed Poisson processes.\\


\noindent{\it Mathematics Subject Classification (MSC):} 60G22, 60G51, 60G55. 

\section{Introduction}
The classical Poisson process is the most useful and popular counting process with index in one dimensional space $\mathbb{R}_+^{1}$. The waiting times in classical Poisson process are exponentially distributed and the mean arrival rate is constant over time. The non-homogeneous Poisson process is obtained by taking the mean arrival rate as a function of time $t.$ The classical Poisson process is a L\'evy process and hence doesn't possess long-range dependence property. The classical Poisson process is generalized in several directions based on inter-arrival times, arrival rate and number of arrivals at an instant. Moreover to model data with stochastic arrival rates, doubly stochastic or time-changed Poisson processes are introduced see e.g. Cox \cite{Cox1955} and Kingman \cite{Kingman1964}. 

The space-fractional Poisson process (SFPP) introduced by Orsingher and Polito \cite{Orsingher2012} is obtained by changing the time of standard Poisson process with a stable subordinator, in which arbitrary number of arrivals occur in small time interval.
The time-fractional Poisson process (TFPP) was introduced by Laskin \cite{Laskin2003} by taking Mittag-Leffler distributed waiting times instead of exponential waiting times in a classical Poisson process. Mittag-Leffler distribution is a heavy tailed distribution due to which TFPP has less number of arrivals than the classical Poisson process on average for sufficiently large $t$. Further TFPP has long-range dependence property (see Biard and Saussereau \cite{Biard2014}). The TFPP is a time-changed Poisson process where time-change is an inverse stable subordinator. Time-changed fractional Poisson process are also recently considered in literature (see e.g. Maheshwari and Vellaisamy \cite{Maheshwari2019, Maheshwari2019a}). 

Another extension of classical Poisson process is to use a general set as an index (see Herbin and Merzbach \cite{Herbin2013}) instead of time or $\mathbb{R}_+^{1}$. In $n$-dimensional Euclidean space $\mathbb{R}^{n}_{+}$ the homogeneous Poisson field is useful model for a random pattern of points
(see Merzbach and Nualart \cite{Merzbach1986}) which is also known as spatial point process having geometric interpretation with application in different areas (see Stoyan et al. \cite{Stoyan1995}).
For example, the spatial point processes are useful in the analysis of observed patterns of points,
where the points represent the locations of disease cases, petty crimes and emergency services etc.

A limitation of standard Poisson process is that there could be at most one arrival at an instant. 
The Poisson distribution of order $k$ (see e.g. Philippou \cite{Philippou1983}) was obtained as a limiting distributions of a sequence of shifted negative binomial distribution of order $k$ Philippou et al. \cite{Georghiou}.  To overcome the limitations of standard Poisson process, Poisson process of order $k$ (PPoK), where at most $k$ arrivals can occur at an instant  was introduced and analyzed as compound Poisson process representation and pure birth process (see e.g.  Kostadinova and Minkova \cite{Kostadinova2012}). The PPoK has applications in standard risk models, where the arrivals of claims in a group of size $k$ could occur. 

Homogeneous Poisson fields are another generalizations of standard Poisson process where index is a general set instead of time(see Merzbach and Nualart \cite{Merzbach1986}, Stoyan et al. and Stoyan et al. \cite{Stoyan1995}). In Leonenko and Merzbach \cite{Leonenko2015}, the author study the fractional Poisson fields which is obtained by time-changing the homogeneous Poisson field of the Euclidean space $\mathbb{R}^{2}_{+}$ with inverse subordinator. 

In this article, we introduce and study time-fractional,  space-fractional Poisson processes of order $k$ on positive real line  $\mathbb{R}_{+}$ and on higher dimension $\mathbb{R}_{+}^{n}$. These processes generalize the Poisson process, homogeneous Poisson field, Poisson process of order $k$, TFPP and SFPP in several directions. 

The rest of the article is organize as follows. In Section 2, we provide all the relevant definitions and results which are used in subsequent sections. In Section 3, we introduce homogeneous Poisson field of order $k$ and their properties.
Section 4 deals with time-fractional Poisson process of order $k$ (TFPPoK). The space-fractional Poisson process of order $k$ (SFPPoK) is introduced in Section 5. To generalize time- and space-fractional Poisson process of order $k$, in Section 6, tempered time-space fractional Poisson process of order $k$ (TTSFPPoK) is given. Section 7 is focused on time- and space-fractional Poisson fields of order $k.$
In the last section Watanabe martingale characterization is given for some time-changed Poisson processes.

\section{Preliminaries} 
In this section, we recall some relevant definitions and properties of homogeneous Poisson field, PPoK, TFPP, SFPP and TTSFPP which will be used in analyzing the fractional Poisson processes and fractional Poisson fields of order $k$.

\subsection{Poisson process of order $k$ (PPoK)}
In this section, we provide important properties of PPoK denoted by $N^{k}(t),\;t\geq 0,$ defined as a compound Poisson process, introduced in Kostadinova and Minkova \cite{Kostadinova2012}. The PPoK $N^{k}(t)$ is given by
\begin{equation}\label{ppok_def}
N^{k}(t) = \sum_{i=1}^{N(t)}X_{i},
\end{equation}
where $X_{i},\; i=1,2,\ldots$ are independent identically distributed (iid) discrete uniform random variables with support $\{1,2,\cdots,k\}$, which are independent to the Poisson process $N(t)$. Further, $N(t)$ is a homogeneous Poisson process with intensity $k \lambda$.
The probability mass function (pmf) is given by 
\begin{equation}\label{pmf_ppok}
    p^{k}(n,t)=\mathbb{P}(N^{k}(t) = n) = \sum_{X \in\Omega(k,n)} e^{-k\lambda t} \frac{(\lambda t)^{\zeta_{k}}}{\Pi_{k}!},
\end{equation}
where $ x_1,\ x_2,\ \dots,\ x_k $ be non-negative integers and $\zeta_{k}  = x_1 + x_2 + \dots + x_k, \; \Pi_{k}! = x_1!x_2!\cdots x_k! $ and
\begin{equation}
    \Omega(k,n) = \{X = (x_1,\ x_2,\ \dots,\ x_k) | x_1 + 2x_2+ \dots + kx_k=n\}.
\end{equation}
\noindent The probability generating function (pgf) $G_{N^{k}}(u,t)$  is given by (see Philippou \cite{Philippou1983})  
\begin{equation}\label{pgf_ppok}
   G_{N^{k}}(u,t) = e^{-\lambda t(k-\sum_{j=1}^{k} u^{j})}.
\end{equation}
\subsection{Homogeneous Poisson field}
 A collection of random variables $\{N( A) : A \in \mathcal{B} \}$ is a Poisson field (see Stoyan et al. \cite{Stoyan1995}) with intensity $\lambda > 0$, where $\mathcal{B}$ is the class of Borel subsets on n-dimensional Euclidean space $\mathbb{R}^{n}_{+}$, if
the following axioms are satisfied:
\begin{itemize}
    \item Only non negative integer values are assumed by $N(A)$ and $\mathbb{P}(N(A)>0)<1$ if $m_{d}(A)>0$.
     \item If $\{A_1, A_2,\ldots,A_n\}$ is a sequence of pairwise disjoint subsets of $\mathcal{B}$, then $\{ N( A_1), N( A_2), \ldots, N(A_n)\}$ is a sequence of
      independent random variables.
     \item $\displaystyle \lim_{m_{d}(A)\rightarrow 0}\frac{\mathbb{P}(N(A)\geq 1)}{\mathbb{P}(N(A)=1)}=1$.
\end{itemize}

 The pmf of homogeneous Poisson field is given by
 \begin{equation}\label{pmf_hpf}
 \mathbb{P}(N(A)=n)= e^{-\lambda m_{d}(A)}\frac{(\lambda m_{d}(A))^{n}}{n!},\;\; n=0,1,2,\ldots,
 \end{equation}
where $m_{d}(A)$ represents the area or volume of $A$, depending on whether $A$ is a region in the plane or a subset of higher dimensional space. The mean and covariance of homogeneous Poisson field are given by (see Leonenko and Merzbach \cite{Leonenko2015})
 \begin{align*}
 \mathbb{E}[N(A)] &= \lambda m_d(A);\\
 \mathrm{Cov}[N(A_1), N(A_2) ] &= \lambda  m_d(A_1 \cap A_2).
 \end{align*}
\subsection{Fractional Poisson processes}
In this section, we recall important properties related to SFPP and TFPP discussed in Orsingher and Polito \cite{Orsingher2012} and Laskin \cite{Laskin2003} respectively.  The TFPP is the generalization of standard Poisson process with Mittag-Leffler waiting times. The TFPP can also be obtained by time-changing the standard Poisson process $N(t)$ by an independent inverse stable subordinator $E_{\beta}(t)$ (see Meerschaert et al. \cite{Meerschaert2011}) such that
\begin{equation}
N_{\beta}(t) = N(E_{\beta}(t)),
\end{equation}
where $E_{\beta}(t) = \inf\{s>0: S_{\beta}(s) > t\}$  and $S_{\beta}(t)$ is the stable subordinator with Laplace transform (LT)
\begin{align}
\mathbb{E}\left(e^{-z S_{\beta}(t)} \right) = e^{-tz^\beta},\; z>0,\; t\geq 0,\; \beta\in(0,1).
\end{align}
The density $f_{\beta}(x,t)$ of stable subordinator $S_{\beta}(t)$ has the following infinite-series form (see Uchaikin and Zolotarev \cite{Uchaikin1999})
\begin{equation}\label{inverse_result}
f_{\beta}(x, t)= t^{-1/\beta}f_{\beta}(t^{-1/\beta}x, 1) =\sum_{k=1}^{\infty}(-1)^{k+1}\frac{\Gamma(\beta k+1)}{k!}\frac{t^k}{x^{\beta k+1}}\frac{\sin(\pi k \beta)}{\pi}=\frac{1}{x}W_{-\beta, 0}(-tx^{-\beta}), \; x>0,
\end{equation}
where $ W_{\gamma, \eta}(z)$ is the Wright's generalized Bessel function (see, e.g.,  Haubold et al. \cite{Haubold2011}).
The process $E_{\beta}(t)$ is non-Markovian with non-stationary increments (see Bingham \cite{Bingham1971}) and has density
\begin{equation}\label{pdf_inverse}
h_{\beta}(x,t) = \frac{x^{-1}}{\beta}W_{-\beta, 0}(-\frac{x}{t^{\beta}}), \; \; x > 0,\;\; t>0,
\end{equation} 
which are not infinitely divisible (see Kumar and Nane \cite{kumar_nane2018}). 
\noindent The pmf $p_{\beta}(n,t) = \mathbb{P}(N_{\beta}(t) = n)$ satisfies the following fractional differential-difference equations with Caputo-Djrbashain fractional derivative in time
\begin{align}\label{derivative in time}
\frac{d^{\beta}}{dt^{\beta}}p_{\beta}(n,t) = -\lambda^{\beta} (p_{\beta}(n,t) - p_{\beta}(n-1,t)) = -\lambda^{\beta} (1-B) p_{\beta}(n,t),
\end{align}
with initial conditions
\begin{equation}
 p_{\beta}(n, 0) =\delta_{n,0}=
          \begin{cases}
                  0,& n\neq 0,\\
                  1, & n=1.
           \end{cases}
\end{equation}
The Caputo-Djrbashain (CD) fractional derivative of order $\beta\in(0,1],$ for a function $g(t),\;t\geq0$ is defined (see Meerschaert and Sikorski \cite{Meerschaert2012}, Sections 2.2, 2.3)) as
\begin{equation}\label{Caputo_Djrbashian_Derivative}
\frac{d^{\beta}}{dt^{\beta}}g(t)= \frac{1}{\Gamma{(1-\beta)}}\int_{0}^{t}\frac{dg(\tau)}{d{\tau}}\frac{d{\tau}}{(t-\tau)^{\beta}},\; \beta\in(0,1].
\end{equation}
\noindent The LT of CD fractional derivative is given by (see e.g., Meerschaert and Sikorski \cite{Meerschaert2012}, p.39)
\begin{align}\label{Caputo_LT}
\mathcal{L}\left(\frac{d^{\beta}}{dt^{\beta}}g(t)\right)=\int_{0}^{\infty}e^{-st}\frac{d^{\beta}}{dt^{\beta}}g(t)dt=s^{\beta}\tilde {g}(s)-s^{\beta-1}g(0^{+}),\; 0<\beta \leq 1,
\end{align}
where $\tilde {g}(s)$ is the LT of the function $g(t),\;t\geq0$, such that
$$
\mathcal{L}(g(t))= \tilde {g}(s)=\int_{0}^{\infty}e^{-st}g(t)dt.
$$
The pgf of TFPP is (see e.g. Laskin \cite{Laskin2003})
\begin{equation}\label{pgftfpp}
G_{N_\beta}(u,t) = M_{\beta,1}(-\lambda t^{\beta}(1-u)) ,\; u \in (0, 1).
\end{equation}
Here $M_{a,b}^c(z)$ is generalized Mittag-Leffler function, introduced by Prabhakar \cite{Prabhakar1971}, 
\begin{equation}
M_{a,b}^c(z) = \sum_{n=0}^{\infty}\frac{(c)_n}{\Gamma(a n + b)}\frac{z^n}{n!}, 
\end{equation}
where $a, b, c \in \mathbb{C}$ with $\mathcal{R}(b) >0$ and $(c)_n$ is Pochhammer symbol. When $c = 1$, it reduces to Mittag-Leffler function.
\begin{align}\label{Mittag}
M_{a,b}(z)=\sum_{k= 0}^{\infty}\frac{z^{k}}{\Gamma(ak+b)},\; z\in \mathbb{C},\; a,b > 0.
\end{align}

Next, we provide the important properties of SFPP. The subordination representation of SFPP is introduced by Orsingher and Polito \cite{Orsingher2012}, denoted as

$$N^{\alpha}(t) =N(S_{\alpha}(t)),$$
where the stable subordinator $S_{\alpha}(t)$ and the homogeneous Poisson process $N(t)$ are assumed to be independent.
The pmf $q^{\alpha}(n,t) = \mathbb{P}(N_{\alpha}(t) = n)$ of the SFPP satisfies the following differential fractional difference equation (see Orsingher and Polito \cite{Orsingher2012})
\begin{align}\label{SFPP}
 \frac{d}{dt}q
 ^{\alpha}(n,t) &= -\lambda^\alpha (1-B)^\alpha q^{\alpha}(n,t),\; n=0,1,2,\ldots 
\end{align}
with initial condition           
\begin{equation}\label{initial-conditions}  
  q^{\alpha}(n,0) = \delta_{n,0}.
\end{equation}
The fractional difference operator $(1-B)^{\alpha}$ is defined as (see Beran \cite{Beran1994})
$$
(1-B)^{\alpha} = \sum_{j=0}^{\infty}{\alpha \choose j} (-1)^j B^j,
$$
where $B$ is a backward shift operator.
\noindent The pgf of SFPP is given by
\begin{equation}\label{pgfsfpp}
G_{N^\alpha}(u,t) =e^{-k^{\alpha} \lambda^{\alpha} t(1-u)^{\alpha}},\; |u|<1,\;\; \alpha \in (0, 1).
\end{equation}

\subsection{Tempered time space fractional Poisson process (TTSFPP)}
The TTSFPP (see Gupta et al. \cite{Gupta2020}) is defined as a time-changed Poisson process
\begin{equation}\label{tempered-TSFPP-subordination}
N^{\alpha, \beta}_{\mu,\nu}(t): = N\left(S_{\alpha,\mu}(E_{\beta,\nu}(t))\right),
\end{equation}
where tempered stable subordinator $S_{\alpha,\mu}(t),\; \alpha \in(0,1),\; \mu>0$ is obtained by exponential tempering in the distribution of stable subordinator (see Rosi\'nski \cite{Rosinski2007}). The inverse tempered stable subordinator (see Kumar and Vellaisamy \cite{Kumar2014}) $E_{\beta,\nu}(t)= \inf\{r>0 : S_{\beta,\nu}(r) > t\},\; \beta \in(0,1),\; \nu>0,$ is the right-continuous inverse of tempered stable subordinator.
The pmf $p^{\alpha, \beta}_{\mu, \nu} (n,t) = \mathbb{P}(N^{\alpha, \beta}_{\mu,\nu}(t) = n)$ of TTSFPP satisfies the following difference-differential equation 
\begin{align}\label{tempered-TSFPP}
\frac{d^{\beta,\nu}}{dt^{\beta,\nu}}p^{\alpha, \beta}_{\mu, \nu} (n,t) = - ((\mu + \lambda(1-B))^{\alpha} - \mu^{\alpha}) p^{\alpha, \beta}_{\mu, \nu} (n,t), 
\end{align}
where $\frac{d^{\beta,\nu}}{dt^{\beta,\nu}}$ is the Caputo tempered fractional derivative of order $\beta\in (0,1)$ with tempering parameter $\nu>0$ is defined by
$$
\frac{d^{\beta,\nu}}{dt^{\beta,\nu}}g(t) =  \frac{1}{\Gamma(1-\beta)}\frac{d}{dt}\int_{0}^{t}\frac{g(u)du}{(t-u)^{\beta}} - \frac{g(0)}{\Gamma(1-\beta)}\int_{t}^{\infty}e^{-\nu r}\beta r^{-\beta-1}dr.
$$
The LT for the Caputo tempered fractional derivative for a function $g(t)$ satisfies
\begin{equation}
\mathcal{L}\left[\frac{d^{\beta,\nu}}{dt^{\beta,\nu}}g\right](s) = ((s+\nu)^{\beta}-\nu^{\beta})\tilde{g}(s) - s^{-1}((s+\nu)^{\beta}-\nu^{\beta}) g(0). 
\end{equation} 
\noindent The pgf of TTSFPP ${N^{\alpha, \beta}_{\mu, \nu}(t)}$ (see Gupta et al. \cite{Gupta2020})
\begin{equation}\label{pgfttsfpp}
G_{N^{\alpha, \beta}_{\mu, \nu}}(u,t) =\sum_{r=0}^{\infty}(-1)^{r}\left[(\mu-k\lambda(1-u))^{\alpha}-\mu^{\alpha}\right]^{r}e^{-t\nu}\sum_{m=0}^{\infty}{\nu}^{m}t^{\beta r +m} M^{r}_{\beta,\beta r+m+1}({t}^{\beta}{\nu}^{\beta}),\; |u|<1.
\end{equation}

\section{Homogeneous Poisson field of order $k$}
In this section, we define and discuss main properties of homogeneous Poisson filed of order $k$.
\begin{definition}
Let $\{N(A) : A \in \mathcal{B} \}$ be homogeneous Poisson field with parameter $k\lambda M_{d}(A)$, where $ m_{d}(A)$ is $d$-dimensional measure defined on  Borel subset $A$ of $\mathbb{R}^{d}$ and $X_i,\;i=1,2,\cdots,$ be the iid discrete uniform random variables such that $\mathbb{P}(X_i = j) = \frac{1}{k},\; j=1,2,\cdots, k$. Then the process $X^k(A)$ defined on subsets of $\mathbb{R}^{d}$
\begin{equation}\label{tfppok}
X^{k}(A) = \sum_{i=1}^{N(A)} X_i,
\end{equation}
is called homogeneous Poisson field of order $k$.
\end{definition}
\noindent The pgf of homogeneous Poisson field of order $k$ is
$$ G_{X^k}(u, A) = \mathbb{E}(u^{X^k(A)})= e^{-k\lambda m_{d}(A) \left(1-\frac{1}{k}\sum_{i=0}^{k} u^{i} \right)}.
$$
Putting $k=1$ gives the pgf of homogeneous Poisson field $N(A)$
$$ G_{N}(u, A) = \mathbb{E}(u^{N(A)})= e^{-k\lambda m_{d}(A) (1-u) }.
$$
The mean and variance of homogeneous Poisson field of order $k$ are given by 
\begin{align*}
    \mathbb{E}[X^k(A)]&=\mathbb{E}[N(A)]\mathbb{E}[X_1]=\frac{k(k+1)}{2}\lambda m_{d}(A)\\
    {\rm Var}[X^k(A)]& = \mathbb{E}[N(A)]\mathbb{E}[X_1^2]=\frac{k(k+1)(2k+1)}{6} \lambda m_{d}(A).
\end{align*}
\begin{proposition}\label{pmf_pfok}
The pmf $p^{k}_{X}(n,A) = \mathbb{P}(X^{k}(A)=n)$ of homogeneous Poisson field of order $k$ is given by
$$
p^{k}_{X}(n,A) = e^{-k\lambda m_{d}(A)} \sum_{X \in \Omega(k,n)}  \frac{(\lambda m_{d}(A))^{\zeta_{k}}}{\Pi_{k}!},
$$
where $ x_1,\ x_2,\ \dots,\ x_k $ be non-negative integers and $\zeta_{k}  = x_1 + x_2 + \dots + x_k, \; \Pi_{k}! = x_1!x_2!\cdots x_k! $ and
\begin{equation}
    \Omega(k,n) = \{X = (x_1,\ x_2,\ \dots,\ x_k) | x_1 + 2x_2+ \dots + kx_k= n\}.
\end{equation}

\end{proposition}

\begin{proof}
The pgf of the homogeneous Poisson field of order $k$,
\begin{align*}
    G_{X^k}(u, A) &= e^{-k\lambda m_{d}(A) + \lambda m_{d}(A) \sum_{i=0}^{k} u^{i}  }\\
    &= e^{-k\lambda m_{d}(A)} \sum_{m=0}^{\infty} \frac{(\lambda m_{d}(A))^m}{m!}(u+u^2+\cdots+u^k)^{m}\\
    & = e^{-k\lambda m_{d}(A)} \sum_{m=0}^{\infty} (\lambda m_{d}(A))^m \sum_{n_1, n_2,\ldots, n_k} \frac{1}{\Pi_{k}!} u^{x_1+2x_2+\cdots+kx_k},
\end{align*}
setting $x_i = n_i$ and $m=n-\sum_{i=1}^{k}(i-1)n_i$, we have
$$
G_{X^k}(u, A) = \sum_{n=0}^{\infty} u^{n} \left(  e^{-k\lambda m_{d}(A)} \sum_{X\in \Omega(k,n)}  \frac{(\lambda m_{d}(A))^{\zeta_{k}}}{\Pi_{k}!}  \right).
$$
The result follows by taking the coefficient of $u^n.$
\end{proof}
\begin{remark}
Putting $k=1$, the homogeneous Poisson field of order $k$ reduces to homogeneous Poisson field.
\end{remark}

\begin{corollary}
Suppose that $A_1, A_2 \in \mathcal{B}$ and $ A_2 \subseteq A_1$, then
\begin{align}\label{cond_hdpp}
\mathbb{P}(X^k(A_2)=m|X^k(A_1)=n) &= \frac{\mathbb{P}(X^k(A_1-A_2)=n-m) \mathbb{P}(X^k(A_2)=m))}{ \mathbb{P}(X^k(A_1)=n)}\nonumber\\
&= \frac{\left(\sum_{X\in \Omega(k,n-m)}  \frac{(\lambda m_{d}(A_1-A_2))^{\zeta_{k}}}{\Pi_{k}!}\right)\left(\sum_{X\in\Omega(k,m)}  \frac{(\lambda m_{d}(A_2))^{\zeta_{k}}}{\Pi_{k}!}\right)}{\sum_{X\in\Omega(k,n)}  \frac{(\lambda m_{d}(A_1))^{\zeta_{k}}}{\Pi_{k}!}}.
\end{align}
\end{corollary}
\begin{remark}
For $k=1$, the conditional distribution of $N(A_2)|\{N(A_1)=n\}$ (see Ross \cite{Ross2009}) is 
$$
\mathbb{P}(N(A_2)=m|N(A_1)=n) = {n \choose m}\left(\frac{ m_{d}(A_2)}{ m_{d}(A_1)}\right)^{m}\left(1-\frac{ m_{d}(A_2)}{ m_{d}(A_1)}\right)^{n-m},
$$
i.e. $N(A_2)|\{N(A_1)=n\} \sim {\rm Binomial}\left(n,\frac{ m_{d}(A_2)}{ m_{d}(A_1)}\right).$
\end{remark}

\section{Time-fractional Poisson process of order $k$ (TFPPoK)}
In this section, we introduce the time-fractional Poisson process of order $k$ (TFPPoK) and discuss its main characteristics.
\begin{definition}[TFPPoK]
Let $N_{\beta}(t, k\lambda)$ be the TFPP with rate parameter $k\lambda > 0$ and $X_i,\;i=1,2,\cdots,$ be the iid discrete uniform random variables such that $\mathbb{P}(X_i = j) = \frac{1}{k},\; j=1,2,\cdots, k$. Then the process defined by
\begin{equation}\label{tfppok}
N_{\beta}^k(t) = \sum_{i=1}^{N_{\beta}(t, k\lambda)} X_i,
\end{equation}
is called the TFPPoK.
\end{definition}
\noindent The pgf $G_{X_1}(u)$ of $X_1$ is
\begin{equation}\label{pgf_disuni}
G_{X_1}(u) = \mathbb{E}(u^{X_1}) = \frac{u}{k}\frac{(1-u^k)}{(1-u)}.
\end{equation}
With the help of \eqref{pgf_disuni} and \eqref{pgftfpp}, the pgf of TFPPoK $N_{\beta}^k(t)$ is given by
\begin{equation}\label{pgf_tfppok}
    G_{N_{\beta}^k}(u,t) = M_{\beta, 1}(-k \lambda t^{\beta}(1-G_{X_1}(u))).
\end{equation}
Using the standard conditioning argument the mean and covariance function of $N_{\beta}^k(t)$ for $s \leq t $ are
$$
\mathbb{E}[N_{\beta}^k(t)] = \frac{k(k+1)}{2}\frac{t^{\beta}}{\Gamma(1+\beta)};$$
$$
\mathrm{Cov} \left(N_{\beta}^k(t), N_{\beta}^k(s)\right) = \frac{k(k+1)(2k+1)}{6} \lambda \mathbb{E}[E_{\beta}(s)]+ \frac{k^2(k+1)^2 \lambda^2}{4} \mathrm{Cov} (E_{\beta}(s), E_{\beta}(t)).
$$
\begin{theorem}
The pmf $p_{\beta}^k(n,t) = \mathbb{P}(N_{\beta}^{k}(t) = n)$ of TFPPoK satisfies the following fractional differential-difference equation
\begin{align}
 \frac{d^{\beta}}{dt^{\beta}}p_{\beta}^k(n,t) &= -k \lambda p_{\beta}^k(n,t) + \lambda\sum_{j=1}^{n\wedge k}p_{\beta}^k(n-j,t),\;\; n= 0, 1,2,\ldots,
\end{align}
with initial conditions $p_{\beta}^k(0,0) = 1$ and $p_{\beta}^k(n,0) = 0$ and $n\wedge k = \min\{k,n\}$.
\end{theorem}
\begin{proof}
 Using $z$-transform (see Debnath and Bhatta \cite{Debnath2014}) in both hand sides, leads to 
\begin{align*}
 \frac{d^\beta}{dt^\beta} \{ \mathcal{Z} p_{\beta}^k(n,t)\}= -k\lambda\left(1-\sum_{j=1}^{k}\frac{z^{-j}}{k}\right) \{ \mathcal{Z} p_{\beta}^k(n,t) \}.
\end{align*}
Next, using the Laplace transform of CD fractional derivative \eqref{Caputo_LT} with respect to the time variable $t$ and the condition $\mathcal{Z}\{p_{\beta}^{k}(n,0) \} =1$, it follows
\begin{align*}
 s^\beta \mathcal{L}[ \mathcal{Z} \{  p_{\beta}^k(n,t) \} ]- s^{\beta-1}= -k\lambda\left(1-G_{X_1}(z^{-1})\right)\mathcal{L}[\mathcal{Z}\{p_{\beta}^k(n,t)\}].
\end{align*}
Thus,
\begin{align*}
\mathcal{L}[\mathcal{Z}\{ p_{\beta}^k(n,t)\}] = \frac{s^{\beta-1}}{s^\beta -k\lambda\left(1-G_{X_1}(z^{-1})\right)}.
\end{align*}
Using the LT of Mittag-Leffler function $\mathcal{L}(M_{\beta,1}(- u t^{\beta})) = \frac{s^{\beta-1}}{u +s^{\beta}}$ (see e.g. Meerschaert and Sikorski \cite{Meerschaert2012}, p. 36), it follows
\begin{align*}
 \mathcal{Z}\{p_{\beta}^k(n,t)\}&= \mathcal{L}^{-1}\left\{\frac{s^{\beta-1}}{s^\beta-k\lambda\left(1-G_{X_1}(z^{-1})\right)}\right\}
  = M_{\beta, 1}(- k\lambda\left(1-G_{X_1}(z^{-1})\right)t^{\beta}),
\end{align*}
which is same as \eqref{pgf_tfppok} by putting $u=z^{-1}$ and hence the result.
\end{proof}

\begin{proposition}
The pmf of TFPPoK is given by
$$
p_{\beta}^k(n,t) = M_{\beta, 1}^{(n)}(-k \lambda t^{\beta})\sum_{X \in \Omega(k,n)} \frac{(\lambda t^{\beta})^{\zeta_{k}}}{\Pi_{k}!}.
$$
\end{proposition}
\begin{proof} Note that
\begin{align*}
  \mathcal{Z}\{p_{\beta}^k(n,t)\}= \sum_{n=0}^{\infty}\frac{(-1)^n(k\lambda)^{n}t^{n\beta}(1-G_{X_1}(z^{-1}))^{n}}{\Gamma(1+n\beta)}.
 \end{align*}
 
To find $p_{\beta}^k(n,t)$, invert the $z$-transform, which is equivalent to finding the coefficient of $z^{-n}$, which leads to
 \begin{align}\label{pmf_tfppok}
     p_{\beta}^k(n,t) & = \sum_{m=0}^{\infty} \frac{(-k\lambda)^{m}t^{\beta m}}{m!}\frac{(n+m)!}{ \Gamma(1+r\beta(n+m))}\sum_{X\in\Omega(k,n)} \frac{(\lambda t^{\beta})^{\zeta_{k}}}{\Pi_{k}!}\nonumber\\
     & = M_{\beta, 1}^{(n)}(-k \lambda t^{\beta})\sum_{X\in \Omega(k,n)} \frac{(\lambda t^{\beta})^{\zeta_{k}}}{\Pi_{k}!},
 \end{align}
where $M_{\beta, 1} (z)$ is the Mittag-Leffler function \eqref{Mittag} evaluated at $z = -k\lambda t^{\beta}$, and 
$M_{\beta, 1}^{(n)}(z)$ is the nth derivative of $M_{\beta, 1} (z)$ evaluated at $z = -k\lambda t^{\beta}$.
\end{proof}

\begin{remark}
For $k=1$, equation \eqref{pmf_tfppok} reduces to the pmf of TFPP (see Beghin and Orsingher,\cite{Beghin2009})
\begin{equation*}
p^{\beta}(k,t) =\frac{(\lambda t^{\beta})^{n}}{n!} \sum_{m=0}^{\infty} \frac{(-\lambda)^{m}t^{\beta m}}{m!}\frac{(n+m)!}{ \Gamma(1+\beta(n+m))}.
\end{equation*}
\end{remark}
 \begin{remark}
For the particular case $\beta = 1$, the pmf of TFPPoK coincide to the PPoK given in equation \eqref{pmf_ppok}.
 \end{remark}
\begin{proposition}
Let $E_{\beta}(t),\; \beta \in(0,1)$, be a right-continuous inverse of stable subordinator and $N^{k}(t),\; t\geq0$, is PPoK with parameter $k\lambda>0$, independent of $E_{\beta}(t).$
Suppose
\begin{equation}
 Z_{\beta}(t) =N^{k}(E_{\beta}(t)),\;  0<\beta<1,
 \end{equation}
then $$ Z_{\beta}(t) \stackrel{d}{=} N^{k}_{\beta}(t).$$
\begin{proof} 
The pgf of $Z_{\beta}(t)$ is given by
$$G_{Z_{\beta}}(u,t) = \mathbb{E}[u^{Z_{\beta}(t)}]= M_{\beta, 1}(- k \lambda t ^{\beta}(1-G_{X_1}(u))$$
which is equal to the pgf of TFPPoK given in \eqref{pgf_tfppok}. Hence by uniqueness of pgf the result follows.
\end{proof}
\end{proposition}

\begin{proposition}\label{ID_fppok}
The marginal distributions of TFPPoK are not infinitely divisible.
\end{proposition}
\begin{proof}
Note that $N_{\beta}(t,k\lambda) \stackrel{d}= N(t^{\beta}E_{\beta}(1))$, where $N(\cdot)$ is Poisson process with parameter $k\lambda.$ Thus
$$
\displaystyle N_{\beta}^k(t) \stackrel{d} = \sum_{j=1}^{N(t^{\beta}E_{\beta}(1))} X_j. 
$$
We have 
\begin{align*}
\displaystyle \frac{\sum_{j=1}^{N(t^{\beta}E_{\beta}(1))}X_j}{t^{\beta}} = \frac{\sum_{j=1}^{N(t^{\beta}E_{\beta}(1))}X_j}{N(t^{\beta}E_{\beta}(1))}\; \frac{N(t^{\beta}E_{\beta}(1))}{t^{\beta}E_{\beta}(1)}\;E_{\beta}(1).
\end{align*}
Using (Th. 3.1.5, p. 81, Mikosch \cite{Mikosch2009}), the first term converges to $\mathbb{E}X_1$ almost surely (a.s) as $t\rightarrow\infty.$ By an application of the renewal theorem, the second term converges to $k\lambda$ almost surely. Thus
$$
\displaystyle \lim_{t\rightarrow\infty}\frac{\sum_{j=1}^{N(t^{\beta}E_{\beta}(1))}X_j}{t^{\beta}} \stackrel{a.s.} \rightarrow k\lambda \mathbb{E}(X_1)E_{\beta}(1).
$$
Hence, 
$$
\displaystyle \lim_{t\rightarrow\infty}\frac{\sum_{j=1}^{N(t^{\beta}E_{\beta}(1))}X_j}{t^{\beta}}  \rightarrow k\lambda \mathbb{E}(X_1)E_{\beta}(1)\;\; \textrm{in distribution}.
$$
This further implies that
$$
\frac{N_{\beta}^k(t)}{t^{\beta}} = \frac{\sum_{j=1}^{N_{\beta}(t, k\lambda)}X_j}{t^{\beta}} \stackrel{d} = \frac{\sum_{j=1}^{N(t^{\beta}E_{\beta}(1))}X_j}{t^{\beta}} \stackrel{d}\rightarrow k\lambda \mathbb{E}(X_1)E_{\beta}(1), \; \textrm{as}\;t\rightarrow \infty.
$$
It is known that the distribution of $E_{\beta}(1)$ is not infinitely divisible (see Kumar and Nane \cite{kumar_nane2018}). Suppose that $N_{\beta}^k(t)$ has infinitely divisible distribution then $\frac{N_{\beta}^k(t)}{t^{\beta}}$ will also have an infinitely divisible distribution (see, e.g., Steutel and Van Harn \cite{Steutel2004}, Prop. 2.1, p. 94). It is known that the limit in distribution of a sequence of random variables with infinitely divisible distributions has an infinitely divisible distribution (see, e.g., Steutel and Van Harn \cite{Steutel2004}, Sato \cite{Sato1999}), we have that the distribution of $E_{\beta}(1)$ is infinitely divisible, and which is a contradiction. Hence the marginal distributions of $N_{\beta}^k(t)$ are not infinitely divisible.
\end{proof}

\section{Space-fractional Poisson process of order k (SFPPoK)}
In this section, we study the SFPPoK and its properties.
\begin{definition}
Let $N^{\alpha}(t, k\lambda)$ be the SFPP and $X_i,\;i=1,2,\cdots,$ be the iid discrete uniform random variables such that $\mathbb{P}(X_i = j) = \frac{1}{k},\; j=1,2,\cdots, k$. Then the process defined by
\begin{equation}\label{fppok}
R_{\alpha}^k(t) = \sum_{i=1}^{N^{\alpha}(t, k\lambda)} X_i,
\end{equation}
is called the SFPPoK.
\end{definition}
\noindent The pgf of SFPPoK $R_{\alpha}^k(t)$ is obtained by using equations \eqref{pgf_disuni} and \eqref{pgfsfpp}
\begin{equation}\label{pgf_sfppok}
    G_{R_{\alpha}^k}(u,t) = e^{-k^{\alpha} \lambda^{\alpha} t(1-G_{X_1}(u))^{\alpha}} ,\; u \in (0, 1).
\end{equation}

\begin{proposition} The alternative characterization of SFPPoK is defined in term of time-changed PPoK such that
$$
R^{\alpha}(t) = N^{k}(S_{\alpha}(t)).
$$
\end{proposition}
\begin{proof}
We have
$$\mathbb{E}[u^{N^{k}(S_{\alpha}(t))}] = e^{-k^{\alpha} \lambda^{\alpha} t\left(1-\frac{1}{k}\sum_{j=1}^{ k}u^{j}\right)^{\alpha}},
$$
which is the same as the pgf of SFPPoK $R^{\alpha}_{k}(t)$. Since SFPPoK and $N^{k}(S_{\alpha}(t))$ both are L\'evy processes, the equivalence follows.
\end{proof}

\begin{theorem}
The pmf $q^{\alpha, k}(n, t) = \mathbb{P}(R^{k}_{\alpha}(t) = n)$ of SFPPoK satisfies the following differential-difference equations 
\begin{align}
 \frac{d}{dt}q^{\alpha, k}(n, t)  &= -k^{\alpha} \lambda^{\alpha}\left(1-\frac{1}{k}\sum_{j=1}^{n \wedge k} B^{n-j}\right)^{\alpha} q^{\alpha, k}(n, t) ,\;\; n=0,1,2,\ldots,
\end{align}
with initial conditions $q^{\alpha, k}(0,0)  = 1$ and $q^{\alpha, k}(n, 0) = 0$.
\end{theorem}
\begin{proof}

\noindent Taking the $z$-transform in both sides, where $\{\mathcal{Z}{q^{\alpha, k}(n, t) }\}$ is the $z$-transform of $q^{\alpha, k}(n, t)$, it follows
\begin{align*}
 \frac{d}{dt}\{\mathcal{Z}{q^{\alpha, k}(n, t) }\} & = -k ^{\alpha} \lambda^\alpha {\left(1-\frac{1}{k}\sum_{j=1}^{k} z^{-j}\right)^\alpha\mathcal{Z} \{q^{\alpha, k}(n, t) \}}].
\end{align*}
Solving the above equation for $\mathcal{Z}{q^{\alpha, k}(n, t) }$ and using initial conditions, leads to
\begin{align*} 
 \{\mathcal{Z}{q^{\alpha, k}(n, t) }\} = \sum_{r=0}^{\infty} \frac{(-1)^{r} (k \lambda)^{\alpha r} t^{r}}{r!} \left(1-\frac{1}{k}\sum_{j=1}^{ k} z^{-j}\right)^{\alpha r}.
\end{align*}
Thus
\begin{align*}
 \{\mathcal{Z}{q^{\alpha, k}(n, t) }\}=\sum_{r=0}^{\infty} \frac{(-1)^{r} (k \lambda)^{\alpha r} t^{r}}{r!} \sum_{n=0}^{\infty} {\alpha r \choose n} \frac{1}{k^{n}}\left(\sum_{j=1}^{ k}z^{-j} \right)^{n},
\end{align*}
which matches with equation \eqref{pgf_sfppok} and hence proved.
\end{proof}

\begin{corollary}
Taking inverse $z$-transform, we get the pmf of SFPPoK
\begin{align}
 q^{\alpha, k}(n, t)=\sum_{r=0}^{\infty} \frac{(-1)^{r+n} (k \lambda)^{\alpha r}t^{r}}{ r!} \sum_{X \in \Omega(k,n)}\frac{\Gamma(\alpha r +1)}{\Gamma(\alpha r -\zeta_{k}+1)} \frac{1}{\Pi_{k}!} \frac{1}{k^{\zeta_{k}}}.
\end{align}
\end{corollary}

\noindent Next, we give an alternative form of marginal pmf of SFPPoK, using a similar argument given in Orsingher et al. \cite{Orsingher2015}.
\begin{corollary}
The marginal pmf of SFPPoK can also be written as 
\begin{align*}
     q^{\alpha, k}(n, t) &= \sum_{r=0}^{\infty} \frac{(-1)^{r+n} (k \lambda)^{\alpha r}t^{r}}{ r!} \sum_{X\in \Omega(k,n)}\frac{\Gamma(\alpha r +1)}{\Gamma(\alpha r -\zeta_{k}+1)} \frac{1}{\Pi_{k}!} \frac{1}{k^{\zeta_{k}}}\\
     &= \sum_{X \in \Omega(k,n)}\frac{(-1)^{n}}{\Pi_{k}! k^{\zeta_{k}} }\sum_{r=0}^{\infty} \frac{(-k^{\alpha} \lambda^{\alpha} t)^{r}}{r!} (\alpha r)(\alpha r-1)\cdots(\alpha r-\zeta_{k} +1)\\
     &= \sum_{X\in \Omega(k,n)}\frac{(-1)^{n}}{\Pi_{k}! k^{\zeta_{k}} } \frac{d^{\zeta_k}}{du^{\zeta_k}} e^{-t k^{\alpha} \lambda^{\alpha} u^{\alpha} } |_{u=1}.
\end{align*}
\end{corollary}

\begin{remark}
Using \cite{Orsingher2015}, for SFPPoK, the probability of jumps in an infinitesimal interval is given by 
\begin{align}
  \mathbb{P}(R^{k}_{\alpha}[t, t+dt) = l) = 
      \begin{cases}
             dt\sum_{X\in \Omega(k,l)}\frac{{\lambda}^{\zeta_{k}}}{\Pi_{k}!} \int_{0}^{\infty} e^{-\lambda k s} s^{\zeta_{k}} \nu(ds) +o(dt), & l\geq 1,\\
                  1-k^{\alpha} \lambda^{\alpha}dt +o(dt) , & l=0.
           \end{cases}
\end{align}
where $\nu(ds) =  \frac{\alpha s^{-\alpha -1}}{\Gamma(1-\alpha)}ds$ is the L\'evy measure of stable subordinator.
\end{remark}

\noindent The following result follows in similar lines of Orsingher et al. \cite{Orsingher2015}.

\begin{proposition}
Let $T_{l}^{\alpha} (t)$ be the first passage time at $l$ of the SFPPoK
$$ T_{l}^{\alpha} (t) = \inf \{t \geq 0 : R^{k}_{\alpha} (t) =l\},\; \; l\geq 0.$$
Then
$$
\mathbb{P}(T_{l}^{\alpha} \in dt)  = \frac{d}{dt} \sum_{j=l}^{\infty}\sum_{X \in \Omega(k,j)}\frac{(-\lambda)^{\zeta_k}}{\Pi_{k}! k^{\zeta_{k}}} \frac{d^{\zeta_k}}{d \lambda^{\zeta_{k}}} e^{-t k^{\alpha}\lambda^{\alpha}}.
$$
\end{proposition}
\begin{proof}
Since $\mathbb{P}(T_{l}^{\alpha} <t) = \mathbb{P}(R^{k}_{\alpha} (t) \geq l)$, we have that
\begin{align*}
    \mathbb{P}(T_{l}^{\alpha} <t) = \sum_{j=l}^{\infty} \int_{0}^{\infty} \sum_{X \in \Omega(k,j)} e^{-k\lambda x} \frac{(\lambda x)^{\zeta_{k}}}{\Pi_{k}!} \mathbb{P}(S_{\alpha}(t) \in dx)
\end{align*}
 and thus, 
\begin{align*}
    \mathbb{P}(T_{l}^{\alpha} \in dt) &= \frac{d}{dt} \sum_{j=l}^{\infty} \int_{0}^{\infty} \sum_{X \in \Omega(k,j)} e^{-k\lambda x} \frac{(\lambda x)^{\zeta_{k}}}{\Pi_{k}!} \mathbb{P}(S_{\alpha}(t)\in dx) \\
    &= \frac{d}{dt}\sum_{j=l}^{\infty}\sum_{X \in \Omega(k,j)}\frac{(-\lambda)^{\zeta_k}}{\Pi_{k}! k^{\zeta_{k}}} \int_{0}^{\infty} \frac{d^{\zeta_k}}{d \lambda^{\zeta_{k}}} e^{-\lambda  k x}\mathbb{P}(S_{\alpha}(t) \in dx)\\
    & = \frac{d}{dt} \sum_{j=l}^{\infty}\sum_{X \in \Omega(k,j)}\frac{(-\lambda)^{\zeta_k}}{\Pi_{k}! k^{\zeta_{k}}} \frac{d^{\zeta_k}}{d \lambda^{\zeta_{k}}} e^{-t k^{\alpha}\lambda^{\alpha}}.
\end{align*}
\end{proof}

 \noindent It is evident that the arrival time of the first event for SFPPoK is exponential, i.e.
 $$ \mathbb{P}(T_{1}^{\alpha} \in dt) = k^{\alpha}\lambda^{\alpha}e^{-t k^{\alpha}\lambda^{\alpha}}.$$
 Further, the time of the second arrival does not follows the gamma (or Erlang) distribution, since
 $$ \mathbb{P}(T_{2}^{\alpha} \in dt) = \lambda^{\alpha} e^{-t k^{\alpha}\lambda^{\alpha}}(k^{\alpha} -\alpha k^{\alpha -1}+ \alpha \lambda^{\alpha}t k^{2\alpha-1}).$$
  The distribution of $T_{l}^{\alpha}$ has the following form
  $$ \mathbb{P}(T_{l}^{\alpha} \in dt) = \mathbb{P}(T_{l-1}^{\alpha} \in dt)  - \sum_{X\in\Omega(k,l-1)}\frac{(-\lambda)^{\zeta_k}}{\Pi_{k}! k^{\zeta_{k}}} \frac{d}{dt}\left(\frac{d^{\zeta_k}}{d \lambda^{\zeta_{k}}} e^{-t k^{\alpha}\lambda^{\alpha}}\right). $$

\begin{proposition}
The marginal distributions of SFPPoK are infinitely divisible.
\end{proposition}
\begin{proof}
 The characteristic exponent of SFPPoK $$\Phi_{t}(\theta)= \log(\mathbb{E}[e^{i\theta{R}^{k}_{\alpha}(t)}])=-k^{\alpha} \lambda^{\alpha} t\left(1-\frac{1}{k}\sum_{j=1}^{ k}(i\theta)^{j}\right)^{\alpha},$$ 
 then
 $$\Phi_{t/n}(\theta) = -k^{\alpha} \lambda^{\alpha} \frac{t}{n}\left(1-\frac{1}{k}\sum_{j=1}^{ k}(i\theta)^{j}\right)^{\alpha}$$
 which implies that
$\Phi_{t}(\theta) = n \Phi_{t/n}(\theta), \; n\geq 1.$
Thus marginal distributions of SFPPoK are Infinite divisible (see Steutel and Van Harn \cite{Steutel2004}).
\end{proof}

\begin{proposition}
The L\'evy density $\nu_{N_{\alpha}^{k}}(x)$ of SFPPoK is
\begin{equation}\label{levy_temp}
\nu_{N_{\alpha}^{k}}(x) = k^{\alpha}\lambda^{\alpha}\sum_{y=1}^{\infty}(-1)^{y+1}\sum_{X\in \Omega(k,y)} \frac{\alpha!}{(\alpha-\zeta_k)! \Pi_k !} \left(\frac{1}{k}\right)^{\zeta_k} \delta_{y}(x) dx.
\end{equation}
\end{proposition}
\begin{proof} We find the L\'evy density with the help of  L\'evy-Khintchine formula (see Sato \cite{Sato1999}),
\begin{align*}
\int_{{\{ 0 \}}^{c}}(e^{i\theta x}-1)\nu_{N_{\alpha}^{k}}(dx)&=\int_{{\{ 0 \}}^{c}}(e^{i\theta x}-1)k^{\alpha}\lambda^{\alpha}\sum_{y=1}^{\infty}(-1)^{y+1}\sum_{X \in \Omega(k,y)} \frac{\alpha!}{(\alpha-\zeta_k)! \Pi_k !} \left(\frac{1}{k}\right)^{\zeta_k} \delta_{y}(x) dx \\
&=k^{\alpha}\lambda^{\alpha}\sum_{y=1}^{\infty}(-1)^{y+1}\sum_{X \in \Omega(k,y)} \frac{\alpha!}{(\alpha-\zeta_k)! \Pi_k !} \left(\frac{1}{k}\right)^{\zeta_k} (\delta_{y}(x)-1)-k^{\alpha} \lambda^{\alpha} +k^{\alpha} \lambda^{\alpha}\\
&=k^{\alpha}\lambda^{\alpha}\sum_{y=0}^{\infty}(-1)^{y+1}\sum_{X \in \Omega(k,y)} \frac{\alpha!}{(\alpha-\zeta_k)! \Pi_k !} \left(\frac{1}{k}\right)^{\zeta_k}e^{i\theta y}\\
& =  -k^{\alpha} \lambda^{\alpha} \left(1-\frac{1}{k}\sum_{j=1}^{ n \wedge k}e^{i\theta j}\right)^{\alpha},
\end{align*}
which is same as the characteristic exponent of SFPPoK given by $\Phi_1(\theta)=-k^{\alpha} \lambda^{\alpha} \left(1-\frac{1}{k}\sum_{j=1}^{ n \wedge k}e^{i\theta j}\right)^{\alpha}$,  hence proved.
\end{proof}

\section{Tempered time-space fractional Poisson process of order $k$ (TTSFPoK)}

In this section, we introduce and study the TTSFPoK.
\begin{definition}
Let $N^{\alpha, \beta}_{\mu,\nu}(t, k\lambda)$ be TTSFFP with rate parameter $k\lambda$ and $X_i,\;i=1,2,\cdots,$ be the sequence of iid discrete uniform random variables such that $\mathbb{P}(X_i = j) = \frac{1}{k},\; j=1,2,\cdots, k$. Also assume that $N^{\alpha, \beta}_{\mu,\nu}(t, k\lambda)$ and $X_i$ are independent. Then the process defined by
\begin{equation}\label{fppok}
Q_{\alpha, \beta}^k(t) = \sum_{i=1}^{N^{\alpha, \beta}_{\mu,\nu}(t, k\lambda)} X_i,
\end{equation}
is called the TTSFPPoK.
\end{definition}

\noindent The pgf of TTSFPPoK $Q_{\alpha, \beta}^k(t)$ is obtained by using equations \eqref{pgfttsfpp} and \eqref{pgf_disuni}, yielding to
\begin{equation}\label{pgf_ttsfppok}
    G_{Q_{\alpha, \beta}^k}(u,t) = \sum_{r=0}^{\infty}(-1)^{r}\left[(\mu-k\lambda(1-G_{X_1}(u)))^{\alpha}-\mu^{\alpha}\right]^{r}e^{-t\nu}\sum_{m=0}^{\infty}{\nu}^{m}t^{\beta r +m} M^{r}_{\beta,\beta r+m+1}({t}^{\beta}{\nu}^{\beta}).
\end{equation}

\begin{theorem}
The pmf $p_{\alpha,\beta}^k(n,t) = \mathbb{P}(Q_{\alpha, \beta}^k(t) = n)$ of TFPPoK satisfies the following fractional differential-difference equation
\begin{align}\label{tempered-TSFPP}
\frac{d^{\beta,\nu}}{dt^{\beta,\nu}}p_{\alpha, \beta}^{k}(n,t) = - \left[ \Bigg\{ \mu + k\lambda \left(1-\frac{1}{k}\sum_{j=1}^{n \wedge k} B^{n-j}\right)\Bigg\}^{\alpha} - \mu^{\alpha}\right]p_{\alpha, \beta}^{k}(n,t), \; n>0. 
\end{align}
\begin{proof}
Suppose $\{\mathcal{Z}p_{\alpha, \beta}^{k}(n,t)\}$ is the $z$-transform of $p_{\alpha, \beta}^{k}(n,t)$, then
\begin{align*}
\frac{d^{\beta,\nu}}{dt^{\beta,\nu}}\{\mathcal{Z}p_{\alpha, \beta}^{k}(n,t)\} = - \left[ \Bigg\{ \mu + k\lambda \left(1-\frac{1}{k}\sum_{j=1}^{k} Z^{-j}\right)\Bigg\}^{\alpha} - \mu^{\alpha}\right]\{\mathcal{Z}p_{\alpha, \beta}^{k}(n,t)\}.
\end{align*}
Using Laplace transform with respect to the time variable $t$ and assuming $|(\mu-k\lambda(1-G_{X_1}(u)))^{\alpha}-\mu^{\alpha}| < |(s+\nu)^{\beta}-\nu^{\beta}|$, it follows
\begin{align*}
\mathcal{L}\left[\{\mathcal{Z}p_{\alpha, \beta}^{k}(n,t)\}\right]=&\frac{1}{s}\left[1+\frac{\left[(\mu-k\lambda(1-G_{X_1}(u)))^{\alpha}-\mu^{\alpha}\right]}{(s+\nu)^{\beta}-\nu^{\beta}}\right]^{-1}\\
= &\sum_{r=0}^{\infty}(-1)^{r}\frac{\left[(\mu-k\lambda(1-G_{X_1}(u)))^{\alpha}-\mu^{\alpha}\right]^{r}}{s((s+\nu)^{\beta}-\nu^{\beta})^{r}}.
\end{align*}
The inverse LT (see Gupta et al. \cite{Gupta2020}) $$ \mathcal{L}^{-1}\left[\frac{1}{s((s+\nu)^{\beta}-\nu^{\beta})^{r}}\right]= \int_{0}^{t}e^{-\nu y} y^{\beta r-1}M^{r}_{\beta,\beta r}(\nu^{\beta} y^{\beta})dy = e^{-t \nu}\sum_{m=0}^{\infty}{\nu}^{m}t^{\beta r+m} M^{r}_{\beta,\beta r+m+1}(\nu^{\beta} t^{\beta}).
$$
Further,
\begin{align*}
\{\mathcal{Z}p_{\alpha,\beta}^{k}(n,t)\}=\sum_{r=0}^{\infty}(-1)^{r}\left[(\mu-k\lambda(1-G_{X_1}(u)))^{\alpha}-\mu^{\alpha}\right]^{r}e^{-t\nu}\sum_{m=0}^{\infty}{\nu}^{m}t^{\beta r +m} M^{r}_{\beta,\beta r+m+1}({t}^{\beta}{\nu}^{\beta}),
\end{align*}
the result follows by putting $u=z^{-1}$.
\end{proof}
\end{theorem}
\noindent Next, we provide the subordination representation of TTSFPPoK.
\begin{proposition}
The marginals of TTSFPPoK can be obtained by subordinating the PPoK $N^{k}(t)$ by an
independent tempered stable subordinator $S_{\alpha, \mu}(t)$ and then by the inverse tempered stable subordinator $E_{\beta,\nu}(t)$,  such that
$$ Z^{\alpha,\beta}_{\mu, \nu}(t) = N^{k}(S_{\alpha, \mu}(E_{\beta,\nu}(t))),\; \alpha, \beta \in (0,1),\; \mu,\nu \geq 0,$$
then $ Z^{\alpha,\beta}_{\mu, \nu}(t) \stackrel{d}{=} Q_{\alpha, \beta}^k(t).$
\end{proposition}
\begin{proof}
The result follows by comparing the pgf of the $N^{k}(S_{\alpha, \mu}(E_{\beta,\nu}(t)))$ and $Q_{\alpha, \beta}^k(t).$
\end{proof}
 \begin{remark}
 For $k=1$, the TTSFPPoK $Q_{\alpha, \beta}^k(t)$ coincide with the TTSFPP $N_{\mu,\nu}^{\alpha,\beta}(t)$ (see  Gupta et al. \cite{Gupta2020}).
 \end{remark}
\section{Fractional Poisson fields of order k}
In this section, we introduce two types of fractional Poisson fields of order $k.$ The space-fractional Poisson field of order $k$ denoted by $N^{k}_{\alpha_1, \alpha_2,\ldots,\alpha_m}(t_1, t_2,\ldots, t_m)$ and the time-fractional Poisson filed of order $k$ denoted by $N^{k}_{\beta_1, \beta_2, \ldots, \beta_m}(t_1, t_2, \dots,t_m).$ These are generalizations of SFPPoK and TFPPoK in higher dimensions. The introduced random fields are defined as follows:
$$
N^{k}_{\alpha_1, \alpha_2,\ldots,\alpha_m}(t_1, t_2,\ldots, t_m) = N^{k}(S_{\alpha_1}^{1}(t_1),  S_{\alpha_2}^{2}(t_2),\dots, S_{\alpha_m}^{m}(t_m) ),
$$
$$
N^{k}_{\beta_1, \beta_2, \ldots, \beta_m}(t_1, t_2, \dots,t_m) = N^{k}(E_{\beta_1}^{1}(t_1),  E_{\beta_2}^{2}(t_2),\dots, \E_{\beta_m}^{m}(t_m)),\; (t_1, t_2, \ldots, t_m)\in \mathbb{R}_{+}^{m},
$$
where Poisson field of order $k$ denoted by  $N^k(t_1, t_2, \cdots, t_m) = N^{k}([0, t_1] \times [0,t_2]\times \cdots \times [0, t_m])$ on  $\mathbb{R}_{+}^{m}$ is independent of the stable subordinators $S_{\alpha_j}^{j}(t_j)$ and inverse stable subordinators $E_{\beta_j}^{j}(t_j)$ with indices $\alpha_j \in (0,1)$ and $\beta_j \in (0,1)$ for $j= 1,2, \ldots,m $. We also assume that all $E_{\beta_j}^{j}(t_j)$ and $S_{\alpha_j}^{j}(t_j)$ are independent for all $j= 1,2, \ldots,m$. For time-fractional random field of order $k=1$, see Leonenko and Merzbach \cite{Leonenko2015}.
\begin{proposition}\label{pmf_tfpfok}
 The pmf of the time-fractional Poisson field of order $k$ is given by
 \begin{align}
    \mathbb{P}(N^{k}_{\beta_1, \beta_2, \ldots, \beta_m}(t_1, t_2, \dots,t_m) = n) &=\sum_{X\in\Omega(k,n)}  \frac{(\lambda )^{\zeta_{k}}}{\Pi_{k}! \beta_1 \beta_2 \cdots \beta_m}
    \int_{0}^{\infty} \int_{0}^{\infty} \cdots\int_{0}^{\infty} e^{-k\lambda x_1x_2\cdots x_m}\times\nonumber\\ &\prod_{j=1}^{m}x_j^{\zeta_{k}-1}W_{-\beta_j, 0}\left(-\frac{x_j}{t_j^{\beta_j}}\right) dx_1dx_2\cdots dx_m, \; n=0,1,2,\ldots.
\end{align}
\end{proposition}
\begin{proof}
By  standard conditional argument and using the equation \eqref{pdf_inverse} and Prop. \ref{pmf_pfok},
$$
\mathbb{P}(N^{k}_{\beta_1, \beta_2, \ldots, \beta_m}(t_1, t_2, \dots,t_m) = n) = \int_{0}^{\infty} \int_{0}^{\infty} \cdots\int_{0}^{\infty} \mathbb{P}(N^{k}(x_1, x_2,\ldots,x_m) = n)\prod_{j=1}^{m} h_{\beta_j}(x_j,t_j)dx_1dx_2,\cdots dx_m,
$$
where $h_{\beta_j}(x_j,t_j)$ are densities of independent inverse stable subordinator $E_{\beta_j}^{j}(t_j)$ in equation \eqref{pdf_inverse}. Using the pmf of $N^{k}(x_1, x_2,\ldots,x_m)$, the equation \eqref{inverse_result} and the relationship  $m_{d}([0,x_1]\times[0,x_2]\times \cdots \times [0,x_m]) = x_1x_2x_3\cdots x_m$, we get the desired result.

\end{proof}
\noindent The mean and variance of $N^{k}_{\beta_1, \beta_2, \ldots, \beta_m}(t_1, t_2, \dots,t_m)$ are given by
\begin{align*}
\mathbb{E}[N^{k}_{\beta_1, \beta_2, \ldots, \beta_m}(t_1, t_2, \dots,t_m)] &= \frac{\lambda k(k+1)}{2} \prod_{j=1}^{m}\frac{t_j^{\beta_j}}{\Gamma(1+\beta_j) }\\
{\rm Var}[N^{k}_{\beta_1, \beta_2, \ldots, \beta_m}(t_1, t_2, \dots,t_m)] &= \frac{k(k+1)(2k+1)\lambda }{6}\prod_{j=1}^{m}\frac{t_j^{\beta_j}}{\Gamma(1+\beta_j) }\\
&+\frac{\lambda^2 k^2(k+1)^2 }{4} \prod_{j=1}^{m} t_j^{2\beta_j}\left(\frac{1}{\prod_{j=1}^{m} \beta_j \Gamma(2\beta_j)}- \frac{1}{(\prod_{j=1}^{m}\beta_j)^2 \prod_{j=1}^{m}\Gamma^2(\beta_j)}\right).
\end{align*}
\begin{remark}
Putting $k=1$ and $m=2$,  the fractional Poisson field of order k $N^{k}_{\beta_1, \beta_2, \ldots, \beta_m}(t_1, t_2, \dots,t_m)$ reduces to fractional Poisson field (see Leonenko and Merzbach \cite{Leonenko2015}).
\end{remark}
\begin{proposition}
The marginals pmf of time-fractional Poisson fields of order $k$ are not infinitely divisible.
\end{proposition}
\begin{proof}
For $m=2$, we have
$$
\displaystyle \lim_{t_1, t_2 \rightarrow\infty} \frac{N^{k}([0, t_1] \times [0,t_2])}{t_1 t_2 }\stackrel{a.s.} \rightarrow \frac{k (k+1)}{2}\lambda.
$$
Also,
$$
\frac{N^{k}([0, E_{\beta_1}^1(t_1)] \times [0,E_{\beta_2}^2(t_2)])}{t_1^{\beta_1} t_2^{\beta_2}} \stackrel{d}= \frac{N^{k}([0, t_1^{\beta_1}E_{\beta_1}^1(1)] \times [0,t_2^{\beta_2}E_{\beta_2}^2(1)])}{t_1^{\beta_1}t_2^{\beta_2}} \stackrel{a.s.} \rightarrow \frac{k (k+1)}{2}\lambda  E_{\beta_1}^{1}(1)E_{\beta_2}^{2}(1),
$$
as $t_1, t_2 \rightarrow \infty.$
Note that $E_{\beta_1}^1(1)$ and $E_{\beta_2}^2(1)$ are not infinitely divisible. Since the product of two non-infinitely divisible random variables is also non-infinitely divisible, using a similar argument as discussed in the Prop. \ref{ID_fppok}, the marginals distribution of $N^{k}_{\beta_1,\beta_2}(t_1,t_2)$ are not infinite divisible. The result follows similarly for general $m.$
\end{proof}

\begin{proposition}\label{pmf_sfpfok}
 The pmf of the $N^{k}_{\alpha_1, \alpha_2,\ldots,\alpha_m}(t_1, t_2,\ldots, t_m)$ is given by
 \begin{align}
    \mathbb{P}(N^{k}_{\alpha_1, \alpha_2,\ldots,\alpha_m}(t_1, t_2,\ldots, t_m) = n) &= \sum_{X\in\Omega(k,n)}  \frac{(\lambda )^{\zeta_{k}}}{\Pi_{k}! }\int_{0}^{\infty} \int_{0}^{\infty} \cdots\int_{0}^{\infty} e^{-k\lambda x_1x_2\cdots x_m}\times\nonumber\\ &\prod_{j=1}^{m}x_j^{\zeta_{k}-1}W_{-\alpha_j, 0}\left(-\frac{t_j}{x_j^{\alpha_j}}\right) dx_1dx_2\cdots dx_m,\; n=0,1,2,\ldots.
\end{align}
\end{proposition}
\begin{proof}
Using the similar approach as in Prop. \ref{pmf_tfpfok}, it follows.
\end{proof}

\section{A martingale characterization}
Let $L_{f}(t)$ be a real valued  driftless subordinator with strictly increasing sample paths. The Laplace transform for $L_{f}(t)$ is of the form 
$$ \mathbb{E}[e^{-s L_{f}(t)}] = e^{-tf(s)}, $$
where
$$ f(s) =  \int_{0}^{\infty} (1-e^{x s}) \nu(dx),\;\; s>0,\; b\geq0,$$
is the integral representation of  Bernstein functions (see Schilling and Song \cite{Schilling2010}). Here the non-negative L\'evy measure $\nu$ on $\mathbb{R}_{+} \cup \{0\}$ satisfies
$$ \int_{0}^{\infty}(x \wedge 1)  \nu(dx) < \infty,  \; \; \nu([0, \infty)) = \infty.
$$
 The inverse subordinator $H_{f}(t)$ is the first exist time of $L_{f}$, defined by,
 $$H_{f}(t) = \inf\{u\geq0 : L_{f}(u) >t\}.$$ 
The process $H_{f}(t)$ is non-decreasing and its sample paths are continuous. Note that since $L_f$ is continuous that means the composition $H_{f}(L_f(x))= x,\; \forall x$ but converse is not true see e.g. Fortelle \cite{Arnaud2015}.
\noindent Next, we discuss some special cases of strictly increasing subordinators. The following subordinators with Laplace exponent denoted by $f(s)$ are very often used in literature.
\begin{equation}\label{Levy_exponent}
f(s) =
         \begin{cases}
                  s^{\alpha},\; 0<\alpha<1, & $(stable subordinator)$;\\
                 \sum_{i=1}^{n}{c_{i}s^{\alpha_{i}}},\;c_{i}\geq{0},\; \sum_{i=1}^{n}{c_{i}}=  1, &$(mixed stable  subordinator)$;\\
                  (s+\mu)^{\alpha}-\mu^{\alpha},\;\mu >0,\; 0<\alpha<1,&$(tempered stable  subordinator)$;\\
                 \sum_{i=1}^{n}{c_{i}((s+\mu_{i})^{\alpha_{i}}-{\mu_{i}}^{\alpha_{i}})},\;c_{i}\geq{0},\; \sum_{i=1}^{n}{c_{i}}=  1, &$(mixture of tempered stable  subordinator)$;\\
                  p\log(1+\frac{s}{\alpha}),\; p>0,\ \alpha >0, &$(gamma subordinator)$;\\
                  
                   \delta(\sqrt{2s+\gamma^2}-\gamma),\; \gamma>0,\; \delta>0, & $(inverse Gaussian subordinator)$.
			\end{cases}
\end{equation}
\noindent Next, we obtain the martingale characterization for time-changed Poisson processes where the time-change is inverse subordinators.
\begin{proposition}\label{Watanabe}
 Let $\{V(t), \; t\geq 0\}$ be a $\{ \mathscr{F}_{t}\}_{t\geq0}$ adapted simply locally finite point process and $L_{f}(t)$ be the strictly increasing subordinators  such that $H_{f}(t) = \inf\{u\geq0 : L_{f}(u) >t\}$ is the inverse of $L_{f}(t)$. Suppose $H_f(t)$ is $L^p$ bounded for some $p>1$. The process $\{V(t)- \lambda H_{f}(t)\}$ is a right continuous martingale with respect to the filtration $\mathscr{F}_{t}= \sigma(V(s), s\leq t) \vee \sigma(H_{f}(s), s\leq t)$ for some $\lambda >0$ iff $V$ is a subordinated Poisson process with time-change $H_f(t)$.
\end{proposition}
\begin{proof}
 The proof follows using a similar argument discussed in Aletti et al. \cite{Aletti2018}.
Suppose the process $\{V(t)- \lambda H_{f}(t)\}$ is $\mathscr{F}_{t}$-martingale and the inverse of $H_f(t)$ are the collections of stopping times for $t \geq 0$ such that
$$D_{f}(t)= \inf\{x\geq0 : H_{f}(x) \geq t\}.$$
Then  $\{V(D_{f}(t))- \lambda H_{f}(D_{f}(t))\}$ is a martingale by optional sampling theorem (see Theorem \rm{6.29} \cite{Kallenberg1997}). Note that since $H_f$ is continuous and increasing then the composition $H_{f}(D_{f}(t))= t,\; \forall\; t$ (see Fortelle \cite{Arnaud2015}). Which implies that  $\{V(D_{f}(t))- \lambda t\}$ is a martingale.  \\
Moreover, as $D_{f}(t)$ is increasing therefore $V(D_{f}(t))$ is a simple point process. By the Watanabe characterization  (see Watanabe \cite{Watanabe1964} and Br\'emaud \cite{Brrmaud1981}), it follows that $V(D_{f}(t)) = N(t) = N(H_{f}(D_{f}(t)))$ is a homogeneous Poisson process with intensity $\lambda >0$. Thus $V(t)= N(H_{f}(t))$ is a subordinated Poisson processes. Conversely, if $V(t)= N(H_{f}(t)),$ then we need to show that   $\{N(H_{f}(t))- \lambda H_{f}(t)\}$ is $\mathscr{F}_{t}$-martingale.
Note that $V\geq0$ and $H_f$ are non-decreasing, and hence the boundedness of $H_f(t)$ in $L^p,\;p>1$ implies that $\{N(H_{f}(t))- \lambda H_{f}(t), 0\leq t\leq T\}$ is uniformly integrable (see, for example, p. \rm{67} \cite{Kallenberg2002}). Since $N(t)-\lambda t$ is a martingale and $H_f(t)$ is a stopping time, the result follows using Doob's optional sampling theorem (see e.g. Theorem \rm{6.29} in \cite{Kallenberg1997}).
\end{proof}

Here, we provide the Watanabe martingale characterization for some well known time-changed Poisson processes where time-change are the inverse subordinators corresponding to the subordinators given in \eqref{Levy_exponent}.
\begin{proposition}
The following results hold:
\begin{enumerate}[a.]
    \item Let $H_f(t) = E_{\alpha}(t)$ be inverse $\alpha$-stable subordinator then $N(E_{\alpha}(t)) - \lambda E_{\alpha}(t)$ is $\{ \mathscr{F}_{t}\}$-martingale iff $N(E_{\alpha}(t))$ is TFPP (see Aletti et al. \cite{Aletti2018}).
    \item Let $H_f(t) = E_{\alpha_1,\alpha_2,\ldots,\alpha_{n}}(t)$ be inverse of mixed stable subordinator then $N(E_{\alpha_1,\alpha_2,\ldots,\alpha_{n}}(t)) - \lambda E_{\alpha_1,\alpha_2,\ldots,\alpha_{n}}(t)$ is $\{ \mathscr{F}_{t}\}$-martingale iff $N(E_{\alpha_1,\alpha_2,\ldots,\alpha_{n}}(t))$ is mixed-fractional Poisson process (see Aletti et al. \cite{Aletti2018}).
    \item Let $H_f(t) = E_{\alpha, \mu}(t)$ be inverse of tempered stable subordinator then $N(E_{\alpha, \mu}(t)) - \lambda E_{\alpha, \mu}(t)$ is $\{ \mathscr{F}_{t}\}$-martingale iff $N(E_{\alpha, \mu}(t))$ is tempered time fractional Poison process (see Gupta et al. \cite{Gupta2020}).
    \item Let $H_f(t) = E_{\alpha_1, \mu_1,\alpha_2,\mu_2,\ldots,\alpha_n, \mu_n}(t)$ be inverse of mixture of tempered stable subordinators (see Gupta et al. \cite {Gupta2021}) then $N(E_{\alpha_1, \mu_1,\alpha_2,\mu_2,\ldots,\alpha_n, \mu_n}) - \lambda E_{\alpha_1, \mu_1,\alpha_2,\mu_2,\ldots,\alpha_n, \mu_n}$ is $\{ \mathscr{F}_{t}\}$-martingale iff \\ $N(E_{\alpha_1, \mu_1,\alpha_2,\mu_2,\ldots,\alpha_n, \mu_n})$ is the mixture tempered time fractional Poison process.
    
    \item Let $H_f(t) = G(t)$ be the inverse of gamma subordinator, then the Poisson process time-changed by $G(t)$ is called here Poisson inverse gamma process. For Poison inverse gamma distribution see Tzougas \cite{Tzougas2020}. The process $N(G(t)) - G(t)$ is a martingale iff N(G(t)) is Poisson inverse Gamma process.
\end{enumerate}
\end{proposition}
\begin{proof}
The tail probabilities of the inverse of above discussed subordinators given in \eqref{Levy_exponent}  decay exponentially for large $x$ and is given by  (see Prop. \rm{2.1} in Kumar and Nane \cite{kumar_nane2018})
\begin{equation*}
  \mathbb{P}(H_{f}(t) >x) \leq ax^{\rho}e^{-b x^{\xi} + cx},\;\; a,\;b,\;c\;>0,\;\xi >1, \; \rho \;\in \mathbb{R}.
\end{equation*} 
This shows that the inverse of these subordinators have finite moments of all orders. Hence these subordinators are $L^p$ bounded for $p>1$ and hence the result follows using Prop. \ref{Watanabe}.
\end{proof}
\section*{Acknowledgments} N.G. would like to thank Council of Scientific and Industrial Research (CSIR), India for supporting her research under the fellowship award number 09/1005(0021)2018-EMR-I. Further, A.K. would like to express his gratitude to Science and Engineering Research Board (SERB), India for the financial support under the MATRICS research grant MTR/2019/000286.
\vone

\end{document}